\documentclass[12pt, one side,emlines]{amsart}
\usepackage[a4paper,margin=1in]{geometry}
\usepackage{fancyhdr}
\usepackage{enumerate}
\usepackage{amsaddr}
\usepackage{amssymb,latexsym,xy,eucal,mathrsfs}
\textwidth=18cm \textheight=27cm \theoremstyle{plain}
\usepackage{tikz}
\usetikzlibrary{positioning}
\theoremstyle{definition}

\newtheorem{theorem}{Theorem}[section]
\newtheorem{thm}[theorem]{Theorem}
\newtheorem{lem}[theorem]{Lemma}

\newtheorem{corollary}[theorem]{Corollary}

\newtheorem{defn}{Definition}[section]

\begin{document}

	\title{On the Characteristic polynomial of ABS Matrix and ABS-Energy of Some Graphs}\maketitle 

\markboth{ Sharad Barde Ganesh Mundhe Mayur Kshirsagar A.N. Bhavale}{On the Characteristic polynomial of ABS Matrix and ABS-Energy of Some Graphs}\begin{center}\begin{large} Sharad Barde$^1,$ Ganesh Mundhe$^2 $  Mayur Kshirsagar$^3$ and   A.N. Bhavale$^4$ \end{large}\\\begin{small}\vskip.1in\emph{
			{\tiny 1. STE's Sinhgad College of Science, Pune-411046, Maharashtra,
			India and  P.E.S. Modern College of Arts, Science and Commerce (Autonomous) Shivajinagar, Pune-411005 India
			2. Army Institute of Technology, Pune-411015, Maharashtra,
			India
			3. D.E.S's Fergusson College (Autonomous), Pune-411004, Maharashtra, India
4. Department of Mathematics P.E.S. Modern College of Arts, Science and Commerce (Autonomous) Shivajinagar, Pune-411005 India}}\\ 
		E-mail: \texttt{ \tiny 1. sharadbardewphd@gmail.com, 2. gmundhe@aitpune.edu.in, \\3. mayur.kshirsagar@fergusson.edu, 4. hodmaths@moderncollegepune.edu.in}\end{small}\end{center}\vskip.2in
	\begin{abstract} For a graph $G$ with $n$ vertices and $m$ edges, Lin \textit{et al.} \cite{Lin} define the \textit{atom--bond sum-connectivity} ($ABS$) matrix of $G$ such that the $(i,j)^{\text{th}}$ entry is 
\[
\sqrt{1 - \frac{2}{d_i + d_j}}
\]
if vertex $v_i$ is adjacent to the vertex $v_j$, and $0$ otherwise. In this article, we determine the characteristic polynomial of the $ABS$ matrix for certain specific classes of graphs. Furthermore, we compute the $ABS$ eigenvalues and the $ABS$ energy for these classes.\end{abstract}\vskip.2in
	\noindent\begin{Small}\textbf{Mathematics Subject Classification (2010)}:
		05C50     \end{Small}\vskip.2in
	
	\vskip 0.15cm
\baselineskip 19truept 
\section{Introduction}
 \indent

 In this paper, all graphs under consideration are assumed to be finite, undirected, and simple. Let $G$ be a graph with vertex set $V(G) = \{v_1, v_2, \dots, v_n\}$ and edge set $E(G) = \{e_{1}, \dots, e_{m}\}$. Denote by $d_i$ the degree of vertex $v_i$ in $G$. If the vertices $v_i$ and $v_j$ are adjacent in $G$, then $v_i v_j \in E(G)$.  For any undefined terms and notation, we refer reader to \cite{Bap, Brouwer AE}.

Several applications of graph invariants can be found in environmental science, pharmacology, chemistry, and related fields (see \cite{Ali, E. Estrada3}). In chemical graph theory,  a \emph{molecular graph} represents a chemical compound by modeling atoms as vertices and chemical bonds between atoms as edges. Topological indices are numerical parameters that encode essential structural information of a molecule in mathematical form. They have become powerful tools in chemical graph theory, particularly in quantitative structure–property relationships (QSPR) and quantitative structure–activity relationships (QSAR), where they are employed to predict physicochemical properties and biological activities. Because of their computational efficiency, discriminating power, and broad applicability, topological indices are extensively used in drug design, nanotechnology, materials science, and environmental chemistry, establishing their importance in both theoretical and applied research.
 
 \indent Gutman and Trinajsti$\acute{c}$ \cite{Gut} studied degree-based graph invariants and introduced the \emph{Zagreb indices}. These invariants depend on the vertex degrees and are defined as  
\[
M_1(G) = \sum_{v_{i}v_{j} \in E(G)} (d_{i} + d_{j}), 
\qquad 
M_2(G) = \sum_{v_{i}v_{j} \in E(G)} d_{i} d_{j}.
\]

The Randi$\acute{c}$ index \cite{Randi} of a graph $G$ is given by  
\[
R(G) = \sum_{v_{i}v_{j} \in E(G)} \sqrt{\frac{1}{d_{i} d_{j}}}.
\]

Many variations of the Randi$\acute{c}$ index have been proposed, such as the harmonic index, the general Randi$\acute{c}$ index, and the atom--bond connectivity ($ABC$) index.  
Fajtlowicz \cite{Fajtl} introduced the \emph{harmonic index} as  
\[
H(G) = \sum_{v_{i}v_{j} \in E(G)} \frac{2}{d_{i} + d_{j}}.
\]
Nikoli\'{c} \emph{et al.} \cite{Nik} introduced the \emph{general Randi$\acute{c}$ index} $R_{-1}(G)$ of a graph $G$, defined by  
\[
R_{-1}(G) = \sum_{v_{i}v_{j} \in E(G)} \frac{1}{d_{i} d_{j}},
\]
which is also referred to as the \emph{modified Zagreb index}.

Estrada \cite{E. Estrada3}  modified the connectivity index and proposed the \emph{atom--bond connectivity} ($ABC$) index, given by  
\[
ABC(G) = \sum_{v_{i}v_{j} \in E(G)} \sqrt{\frac{d_{i} + d_{j} - 2}{d_{i} d_{j}}}.
\]
Estrada \emph{et al.} \cite{Est1} showed that the index $ABC$ can be used to predict the heat of alkane formation.
In 2022, Ali \emph{et al.} \cite{Ali} proposed a new topological index called the \emph{atom--bond sum-connectivity} ($ABS$) index, defined as  
\[
ABS(G) = \sum_{v_{i}v_{j} \in E(G)} \sqrt{\frac{d_{i} + d_{j} - 2}{d_{i} + d_{j}}}
= \sum_{v_{i}v_{j} \in E(G)} \sqrt{1 - \frac{2}{d_{i} + d_{j}}}.
\]

Aarthi K et al. \cite{Aarthi K} have investigated the highest and lowest values of the atom-bond sum-connectivity (ABS) index within the class of bicyclic graphs.  
Li A et al. \cite{Ali A2}  investigated the extremal values and bounds of this index, particularly for different types of graphs. 
Hu and Wang \cite{Hu Y} obtained the tree class with the highest atom-bond sum-connectivity (ABS) index for a given number of vertices and leaves.   \cite{Li F} focuses on determining the maximum atom-bond sum-connectivity index (ABS index) for graphs with specific parameters, such as clique number and chromatic number. Li F and Ye Q \cite{Li F1} focus on identifying the graphs with the most extreme values (maximum or minimum) of the general Atom-Bond Sum-Connectivity (ABS) index, given specific parameters such as the chromatic number, connectivity, or matching number. 
Researchers have shown considerable interest in the $ABS$ index. For mathematical properties and chemical applications of the $ABS$ index and its variants, see \cite{Aarthi K, Ali A2, Hu Y,  Nithya P, Zhang Y, Zuo X} and the references therein. 

The \emph{adjacency matrix} of $G$, denoted by $A(G)$, is an $n \times n$ matrix $[a_{ij}]_{n\times n}$ whose $(i,j)^{\text{th}}$ entry is given by  
\[
a_{ij} =
\begin{cases}
1, & \text{if } v_{i}v_{j} \in E(G),\\
0, & \text{otherwise}.
\end{cases}
\]
If $\lambda_{1}, \lambda_{2}, \dots, \lambda_{n}$ are the eigenvalues of $A(G)$, then the \emph{energy} $E_A(G)$ of $G$ is defined as  
\[
E_A(G) = \sum_{i=1}^{n} |\lambda_{i}(G)|.
\]

The \emph{incidence matrix} of $G$, denoted by $F(G)$, is an $n \times m$ matrix $[f_{ij}]_{n\times m}$ whose $(i,j)^{\text{th}}$ entry is  
\[
f_{ij} =
\begin{cases}
1, & \text{if vertex } v_{i} \text{ is incident to edge } e_j,\\
0, & \text{otherwise}.
\end{cases}
\]

From the $ABS$ index of $G$, Lin \emph{et al.} \cite{Lin} defined the \emph{$ABS$ matrix} of $G$ as the $n \times n$ matrix $[(abs)_{ij}]_{n\times n}$ whose $(i,j)^{\text{th}}$ entry is  
\[
(abs)_{ij} =
\begin{cases}
\sqrt{\frac{d_{i} + d_{j} - 2}{d_{i} + d_{j}}}, & \text{if } v_{i}v_{j} \in E(G),\\
0, & \text{otherwise}.
\end{cases}
\]
We denote the $ABS$ matrix of $G$ by $\tilde{A}(G)$, and its eigenvalues by $\mu_{1}, \mu_{2}, \dots, \mu_{n}$. The \emph{$ABS$ energy} of $G$ is then defined as  
\[
E_{\tilde{A}}(G) = E_{ABS}(G) = \sum_{i=1}^{n} |\mu_{i}(G)|.
\]

In this paper, we investigate properties of the $ABS$ matrix, $ABS$ eigenvalues, and $ABS$ energy of graphs. Section~2 contains the necessary preliminaries. In Section~3, we establish several properties of the $ABS$ eigenvalues. In Section~4, we derive results concerning the $ABS$ energy of graphs.

 \indent
  
\section {Characteristic Polynomial of ABS Some Specific Class of Graphs}
Let $G$ be a graph. Denote by $\psi$ the characteristics polynomial of a graph $G$ called as \textit{adjacency polynomial} defined as $\psi(G:\mu)=\det(\mu I-A(G)),$ where $A(G)$ is an adjacency matrix of $G.$ Denote by $\phi$ the characteristic polynomial of an $ABS$ matrix  of $G$ called as  \textit{  ABS polynomial} defined as $\phi(G:\mu)=\det(\mu I-\tilde{A}(G)),$ where $\tilde{A}$ is an $ABS$ matrix of $G.$ Let $K_n$, $C_n$ and $S_n$ denote a complete, cycle and star graph on $n$ vertices. The complete bipartite graph is denoted by $K_{m,n}$ on $m+n$ vertices.  
\par Recall the following results in sequel.
\begin{lem}  [Cvetkovi$\acute{c}$, Doob and  Sach.\cite{Doob}] \label{5}
Let $G$ be a graph on $n$ vertices and $m$ edges and $L(G)$ be line graph of $G$. If $F(G)$ is incidence matrix of $G$.
\begin{itemize}
\item[(i)] If $G$ is $r$- regular graph then $F(G) F(G)^{t} = A(G) + rI_{n}$ and
\item[(ii)] $F(G)^{t} F(G) = 2I_{n} + A(L(G))$.
\end{itemize}
\end{lem} 
\begin{lem}  [Cvetkovi$\acute{c}$, Doob and  Sach.\cite{Doob}] \label{6} 
Let $M$ and/or $Q$ be an invertible matrix, then
${\left| \begin{array}{cc}
M & N \\
P & Q \end{array} \right|}$ = $|M| |Q - P M^{-1} N| = |Q| |M - N Q^{-1} P|$.
\end{lem}
\par In the following result, we establish the relation  between  the adjacency polynomial and $ABS$ polynomial for a $r$- regular graph. 
\begin{thm} \label{7}
	Let $G$ be a $r$-regular graph, then 
	$\displaystyle{\phi\big( G : \mu \big ) = \left(\frac{\sqrt{r^{2}-r}}{r}\right)~\psi \left(G : \frac{\mu}{\frac{\sqrt{r^{2}-r}}{r}} \right)}.$
\end{thm}
\begin{proof}
	We have, $\tilde{A}(G)$ $=$ $\dfrac{\sqrt{r^{2}-r}}{r}$ $A(G)$.
	Then 
	\begin{eqnarray*}
		\phi (G : \mu )& =& det\left(\mu I - \tilde{A}(G)\right)\\
		&=& \left|\mu I - (\frac{\sqrt{r^{2}-r}}{r}) A(G)\right|\\
		&=& \left(\frac{\sqrt{r^{2}-r}}{r}\right) \left| \left(\frac{\mu}{\frac{\sqrt{r^{2}-r}}{r}}\right) I - A(G) \right|\\
		&=& \left(\frac{\sqrt{r^{2}-r}}{r}\right)\psi \left(G : \dfrac{\mu}{\frac{\sqrt{r^{2}-r}}{r}} \right).  
	\end{eqnarray*}
\end{proof}
Recall the following definition.
\begin{defn} [Harary \cite{Harry}]
	Let $G$ be a graph on $n$ vertices and $m$ edges. Then the \textit{subdivision graph} of $G$ is the graph obtained by inserting a new vertex on each edge of $G$. 
\end{defn}
\par Denote by $S(G)$ the subdivision graph of $G.$ Observe that, the total numbers of vertices and edges in $S(G)$ are $n + e$ and $2m$ respectively. If $x$ is a vertex of $G$ then $d_{x}S(G)=d_{x}(G)$ and if $y$ is inserting a new vertex on each edge of $G$ then $d_{y}S(G) = 2$.  
\par In the following result, we establish the relation  between  the adjacency polynomial  and $ABS$ polynomial of subdivision graph $S(G)$ for a $r$-regular graph. 
\begin{thm} \label{8}
	Let $G$ be a $r$-regular graph on $n$ vertices and $m$ edges. Then 
	\begin{align*} \displaystyle{\phi\big(S(G) : \mu \big ) =  \left(\frac{r}{r+2}\right) \mu^{m-n}~\psi \left(G : \frac{{\mu}^{2}(r+2) - r^{2}}{r} \right).}
	\end{align*}
\end{thm}
\begin{proof}
	Let $G$ be a $r$-regular graph on $n$ vertices and $m$ edges. Under labelling the vertices of $S(G)$, the $ABS$ matrix of $S(G)$ is 
	\begin{center}
		$\displaystyle{\tilde{A}(S(G))={\left[ \begin{array}{cc}
					O_{m} & {\sqrt{\dfrac{r}{r+2}} F(G)_{m \times n}^{t}} \\
					{\sqrt{\dfrac{r}{r+2}} F(G)_{n \times m}}  & O_{n} \end{array} \right]}.}$ 
	\end{center}
	Where,  $O$ is a matrix whose all entries are $0$ and  $F(G)$ is an incidence matrix of $G$. 
	Then, the characteristic polynomial of $\tilde{A}(S(G))$ is 
	\begin{center}
		$\displaystyle{\phi(S(G) : \mu)= det(\mu I - \tilde{A}(S(G)))=  {\left| \begin{array}{cc}
					\mu I_{m} & -{\sqrt{\dfrac{r}{r+2}} F(G)_{m \times n}^{t}} \\
					-{\sqrt{\dfrac{r}{r+2}} F(G)_{n \times m}} & \mu I_{n} \end{array} \right|.}}$ 
	\end{center}
	By Lemma {\ref{6}},
	\begin{eqnarray*}
		\phi(S(G) : \mu) &=& \mu^{m} \left|\mu I_{n} - \left({\dfrac{r}{r+2}}\right)\dfrac{F(G)_{n \times m} I_{m} F(G)_{m \times n}^{t}}{\mu}\right|
	\end{eqnarray*}
	Now, by Lemma {\ref{5}} we have,
	\begin{eqnarray*}
		\phi(S(G) : \mu) &=& \mu^{m-n} \left|\mu^{2} I_{n} - \left({\dfrac{r}{r+2}}\right) (A(G) + r I_{n})\right|\\
		&=& \mu^{m-n} \left|\mu^{2} I_{n} - {\dfrac{r}{r+2}} A(G)-\dfrac{r^2}{r+2} I_{n}\right|\\
		&=& \mu^{m-n} \left|\left(\mu^{2} -\dfrac{r^2}{r+2}\right) I_{n}- {\dfrac{r}{r+2}} A(G)\right|\\
		&=& \left({\dfrac{r}{r+2}}\right)\mu^{m-n}  \left|\dfrac{\left(\mu^{2} - {\dfrac{r^2}{r+2}}\right)} {\left({\frac{r}{r+2}}\right)}I_{n} - A(G)\right|\\
		&=& \left({\dfrac{r}{r+2}}\right)\mu^{m-n}  \left|\left(\dfrac{\mu^2(r+2)-r^2}{r}\right)I_{n} - A(G)\right|\\
		&=&\left({\dfrac{r}{r+2}}\right) \mu^{m-n} \psi \left(G : \dfrac{{\mu}^{2}(r+2) - r^{2}}{r} \right).
	\end{eqnarray*}
\end{proof}
Recall the following definition.
\begin{defn} [Sampathkumar and  Chikkodimath \cite{Samp}]
	Let $G$ be a graph on $n$ vertices and $m$ edges. Then the \textit{semi total point graph} of $G$  denoted by $T_{1}(G)$ is a graph whose vertex set is the union of vertex set and edge set of $G$. The two vertices are adjacent in $T_{1}(G)$ if they are adjacent vertices in $G$ or one is vertex and other is an edge incident on it. 
\end{defn}
\par If $x$ is a vertex of $G$ then $d_{x}T_{1}(G) = 2d_{x}(G)$ and if $e$ is an edge in $G$, then $d_{e}T_{1}(G) = 2$.  
\par In the following result, we establish the relation  between  the adjacency polynomial  and $ABS$ polynomial of semitotal graph $T_1(G)$ for a $r$-regular graph.
\begin{thm} \label{9}
	Let $G$ be a $r$-regular graph on $n$ vertices and $m$ edges. Then 
	\begin{align*}
		\phi\big(T_1(G) : \mu \big ) =  \left[\sqrt{\dfrac{2r-1}{2r}} \mu + \dfrac{r}{r+1}\right] \mu^{m-n}\psi \left(G : \frac{{\mu}^{2}-\frac{r^2}{r+1}}{\left(\sqrt{\frac{2r-1}{2r}}\right)\mu + \frac{r}{r+1}} \right).
	\end{align*}
\end{thm}
\begin{proof}
	Let $G$ be a $r$-regular graph on $n$ vertices and $m$ edges. Then the  $ABS$ matrix of $T_1(G)$ is 
	\begin{center}
		$\tilde{A}(T_1(G))=\displaystyle{{\left[ \begin{array}{cc}
					O_{m} & \sqrt{\dfrac{r}{r+1}} F(G)_{m \times n}^{t} \\
					{\sqrt{\dfrac{r}{r+1}} F(G)}_{n \times m}  & \sqrt{\dfrac{2r-1}{2r}} A(G)_{n} \end{array} \right]},}$ 
	\end{center}
	Where,  $O$ is a matrix whose all entries are $0$ and  $F(G)$ is an incidence matrix of $G$. 
	The characteristic polynomial of $\tilde{A}(T_1(G))$ is 
	\begin{center}
		$\phi((T_1(G)) : \mu)= det(\mu I - \tilde{A}(T_1(G)))={\left| \begin{array}{cc}
				\mu I_{m} & -\sqrt{\dfrac{r}{r+1}} F(G)_{m \times n}^{t} \\
				-\sqrt{\dfrac{r}{r+1}} F(G)_{n \times m} & \mu I_{n} - \sqrt{\dfrac{2r-1}{2r}} A(G) \end{array} \right|.}$ 
	\end{center}
	By Lemma {\ref{6}},
	\begin{eqnarray*}
		\phi(T_{1}(G) : \mu)& = &\mu^{m} \left|\left(\mu I_{n} -{\sqrt{\frac{2r-1}{2r}}} A(G)\right)- \dfrac{r}{r+1}\frac{F(G)_{n\times m} I_{m} F(G)_{m \times n}^{t}}{\mu}\right|\\
		&= &\mu^{m-n} \left|\mu^{2} I_{n} - \sqrt{\dfrac{2r-1}{2r}}\mu A(G) - \dfrac{r}{r+1}{F(G) F(G)^{t}}\right|.
	\end{eqnarray*}
	Now, by Lemma {\ref{5}} we have,
	\begin{eqnarray*}
		\phi(T_{1}(G) : \mu)&=& \mu^{m-n} \left|\mu^{2} I_{n} - \sqrt{\frac{2r-1}{2r}}\mu A(G) - \dfrac{r}{r+1}(A(G) + r I_{n})\right|\\
		&=& \mu^{m-n} \left|\mu^{2} I_{n} - \sqrt{\frac{2r-1}{2r}}\mu A(G) - \dfrac{r}{r+1}A(G) -  \frac{r^2}{r+1}I_{n}\right|\\
		&=& \mu^{m-n} \left|\left(\mu^{2} - \frac{r^2}{r+1}\right) I_{n}-\left( \sqrt{\frac{2r-1}{2r}}\mu + \frac{r}{r+1}\right)A(G)\right|\\
		&=& \mu^{m-n} \left[({\sqrt{\frac{2r-1}{2r}}})\mu + \dfrac{r}{r+1}\right] \left|\left(\frac {\mu^{2} -\frac{r^2}{r+1}} {{\sqrt{\frac{2r-1}{2r}}}\mu+\frac{r}{r+1}}\right) I_{n} - A(G)\right|\\
		&=&  \mu^{m-n} \left[({\sqrt{\frac{2r-1}{2r}}})\mu + \dfrac{r}{r+1}\right]\psi \left(G : \frac{{\mu}^{2}-\frac{r^2}{r+1}}{\left(\sqrt{\frac{2r-1}{2r}}\right)\mu + \frac{r}{r+1}} \right).
	\end{eqnarray*}
\end{proof}
Recall the following definition.
\begin{defn} [Sampathkumar and  Chikkodimath \cite{Samp}]
	Let $G$ be a graph on $n$ vertices and $m$ edges. The \textit{semitotal line graph} of $G$ is denoted by $T_{2}(G)$, whose vertex set is the union of vertex set and edge set of $G$. The two vertices are adjacent in $T_{2}(G)$ if they are adjacent edges in $G$ or one is a vertex and other is an edge incident on it.
\end{defn}
\par  If $x$ is a vertex of $G$ then $d_{x}T_{2}(G) = d_{x}(G)$ and if $e = uv \in E(G)$ is an edge in $G$, then $d_{e}T_{2}(G) = d_u + d_v$.   In the following result, we establish the relation  between  the adjacency polynomial  and $ABS$ polynomial of semitotal line graph $T_2(G)$ for a $r$-regular graph.
\begin{thm} \label{10}
	Let $G$ be a $r$-regular graph on $n$ vertices and $m$ edges. Then 
	\begin{align*}
		\phi\big(T_2(G) : \mu \big ) =  \left[\sqrt{\dfrac{4r-2}{4r}} \mu + \frac{3r-2}{3r}\right] \mu^{n-m}\psi \left(G(L) : \dfrac{{\mu}^{2}-\frac{6r-4}{3r}}{\sqrt{\frac{4r-2}{4r}}\mu + \frac{3r-2}{3r}} \right).
	\end{align*}
\end{thm}
\begin{proof}
	Let $G$ be a $r$-regular graph on $n$ vertices and $m$ edges.
	The $ABS$ matrix of $T_2(G)$ is 
	\begin{center}
		$\tilde{A}(T_2(G))$ $=$  ${\left[ \begin{array}{cc}
				\sqrt{\dfrac{4r-2}{4r}} A(L(G))_{ m} & \sqrt{\dfrac{3r-2}{3r}} F(G)_{m \times n}^{t} \\
				{\sqrt{\dfrac{3r-2}{3r}}} {F(G)_{n \times m}}  & O_{n} \end{array} \right]},$ 
	\end{center}
	where $A(L(G)$ is the adjacency matrix of line graph of $G, O$ is a matrix whose all entries are $0$ and  $F(G)$ be an incidence matrix of $G$.
	Now, the characteristic polynomial of $\tilde{A}(T_2(G))$ is 
	\begin{center}
		$\phi((T_2(G)) : \mu)=det(\mu I - \tilde{A}(T_2(G))) ={\left| \begin{array}{cc}
				\mu I_{m} - \sqrt{\dfrac{4r-2}{4r}} A(L(G)) & -\sqrt{\dfrac{3r-2}{3r}} F(G)_{m \times n}^{t} \\
				-\sqrt{\dfrac{3r-2}{3r}} F(G)_{n \times m} & \mu I_{n} \end{array} \right|.}$ 
	\end{center}
	By Lemma {\ref{6}}, we have
	\begin{eqnarray*}
		\phi((T_2(G) : \mu)& =& \mu^{n} \left|\left(\mu I_{m} - \sqrt{\dfrac{4r-2}{4r}} A(L(G))\right) - \dfrac{3r-2}{3r} \dfrac{F(G)_{m \times n}^{t} I_nF(G)_{n \times m}}{\mu}\right|\\
		&=& \mu^{n-m} \left|\left(\mu^{2} I_{m} - \mu \sqrt{\frac{4r-2}{4r}} A(L(G))\right) - \dfrac{3r-2}{3r} {F(G)_{m \times n}^{t} F(G)_{n \times m}}\right|.
	\end{eqnarray*}
	By Lemma {\ref{5}},
	\begin{eqnarray*}
		&=& \mu^{n-m} \left|\left(\mu^{2} I_{m} - \mu \sqrt{\dfrac{4r-2}{4r}} A(L(G))\right) - \dfrac{3r-2}{3r} \left(2I_{m} + A(L(G))\right)\right|\\
		&=& \mu^{n-m} \left|\left(\mu^{2} I_{m} - \mu \sqrt{\frac{4r-2}{4r}} A(L(G))\right) -\dfrac{2(3r-2)}{3r} I_{m} - \dfrac{3r-2}{3r} A(L(G)))\right|\\
		&=& \mu^{n-m} \left|\left(\mu^{2} - \dfrac{6r-4}{3r}\right) I_{m} -\left(\mu \sqrt{\frac{4r-2}{4r}} + \dfrac{3r-2}{3r}\right) A(L(G))\right|\\
		&=& \mu^{n-m} \left[\mu \sqrt{\dfrac{4r-2}{4r}} + \dfrac{3r-2}{3r}\right]         \left|\left(\dfrac{\mu^{2} - \frac{6r-4}{3r}} {\mu \sqrt{\frac{4r-2}{4r}} + \frac{3r-2}{3r}}\right) I_{m} - A(L(G))\right|\\
		&=& \mu^{n-m} \left[\mu \sqrt{\frac{4r-2}{4r}} + \frac{3r-2}{3r}\right]\psi \left(G(L) : \dfrac{{\mu}^{2}-\frac{6r-4}{3r}}{\sqrt{\frac{4r-2}{4r}}\mu + \frac{3r-2}{3r}} \right).
	\end{eqnarray*}
\end{proof}
\begin{thm} \label{13}
Let $P_n$ be a path graph on $n$ vertices for $n \geq 5$. Then the $ABS$ characteristic polynomial of $P_n$ is  
$\phi\big(P_n : \mu \big ) = \mu^{2} \varOmega_{n-2} -\frac{2}{3} \mu \varOmega_{n-3} + \frac{1}{9} \varOmega_{n-4},$ where  $\varOmega_{1} = \mu, \varOmega_{2} = \mu^{2} - \frac{1}{2}$ and
 $\varOmega_{m} = \mu \varOmega_{m-1} - \frac{1}{2} \varOmega_{m-2},$ for  $m \geq 3$.  
\end{thm}
\begin{proof}
Let $P_n$ be a path graph on $n$ vertices for $n \geq 5$. Then the $ABS$ matrix of $P_n$ is
\begin{center}
 $\tilde{A}(P_n) = \left( \begin{array}{cccccc}
0 & \sqrt{{\frac{1}{3}}} & 0 & 0 & \dots & 0 \\ 
\sqrt{{\frac{1}{3}}} & 0 & \sqrt{{\frac{1}{2}}} & 0 & \dots &0\\
0 & \sqrt{{\frac{1}{2}}} & 0 & \sqrt{{\frac{1}{2}}} & 0 &\vdots  \\
 \vdots& 0 & \sqrt{{\frac{1}{2}}} & 0  &\ddots  &0 \\
 0 & \vdots & 0& \ddots &  \ddots & \sqrt{{\frac{1}{3}}} \\
0& 0 &\dots &0 &\sqrt{{\frac{1}{3}}} & 0\\
\end{array} \right)_{n}.$
\end{center}
Therefore, the characteristic polynomial of $\tilde{A}(P_n)$ is  
\begin{center}
	$\phi\big(P_n:\mu) = det(\mu I - \tilde{A}(P_n)) = \left| \begin{array}{cccccc}
\mu & -\sqrt{{\frac{1}{3}}} & 0 & 0 & \dots & 0 \\ 
-\sqrt{{\frac{1}{3}}} & \mu & -\sqrt{{\frac{1}{2}}} & 0 & \dots &0\\
0 & -\sqrt{{\frac{1}{2}}} & \mu & -\sqrt{{\frac{1}{2}}} & 0 &\vdots  \\
 \vdots& 0 & -\sqrt{{\frac{1}{2}}} & \mu  &\ddots  &0 \\
 0 & \vdots & 0& \ddots &  \ddots & -\sqrt{{\frac{1}{3}}} \\
0& 0 &\dots &0 & -\sqrt{{\frac{1}{3}}} & \mu\\
\end{array} \right|_{n}.$
\end{center}
For any $m\geq 3$, consider
$C_m$
$= \left( \begin{array}{cccccc}
\mu & -\sqrt{{\frac{1}{2}}} & 0 & 0 & \dots & 0 \\ 
-\sqrt{{\frac{1}{2}}} & \mu & -\sqrt{{\frac{1}{2}}} & 0 & \dots &0\\
0 & -\sqrt{{\frac{1}{2}}} & \mu & -\sqrt{{\frac{1}{2}}} & 0 &\vdots  \\
 \vdots& 0 & -\sqrt{{\frac{1}{2}}} & \mu  &\ddots  &0 \\
 0 & \vdots & 0& \ddots &  \ddots & -\sqrt{{\frac{1}{2}}} \\
0& 0 &\dots &0 & -\sqrt{{\frac{1}{2}}} & \mu\\
\end{array} \right)_{m}.$\\
Let $\varOmega_{m} = \det(C_m).$ Observe that, $\varOmega_{m} = det~(C_{m}) = \mu \varOmega_{m-1} - \frac{1}{2} \varOmega_{m-2}$.  
Therefore,\begin{eqnarray*}
 \phi\big(P_n : \mu \big )&=&\left| \begin{array}{ccccccc}
\mu & \vline~~~~ &-\sqrt{{\frac{1}{3}}} & 0 & \dots &\vline~~~~& 0\\ 
\hline
-\sqrt{{\frac{1}{3}}} &\vline~~~~ & &  &  &\vline~~~~&0\\
0 & \vline~~~~  & & C_{n-2} &  &\vline~~~~&0\\
 \vdots& \vline~~~~ & &   & & \vline~~~~ &\vdots\\
 0 &\vline~~~~  & & & & \vline~~~~& -\sqrt{{\frac{1}{3}}}\\
\hline
 0& \vline~~~~ &0  &\dots &-\sqrt{{\frac{1}{3}}} & \vline~~~~& \mu\\
\end{array} \right|_{n.}
\end{eqnarray*}
$= \mu
\left| \begin{array}{cccccc}
	& &  &  &\vline& 0\\ 
	& & C_{n-2} &  &\vline&0\\
	& &   & & \vline &\vdots\\
	& & & & \vline& -\sqrt{{\frac{1}{3}}}\\
	\hline
	0 &0  &\dots &-\sqrt{{\frac{1}{3}}} & \vline& \mu\\
\end{array} \right|_{n-1}+\sqrt{{\frac{1}{3}}} 
\left| \begin{array}{ccccccc}
-\sqrt{{\frac{1}{3}}}&\vline &-\sqrt{{\frac{1}{2}}}  &0 & \dots &\vline& 0\\ 
\hline
0 &\vline& & C_{n-3} &   &\vline&0\\
\vdots &\vline& &  & & \vline &\vdots\\
0&\vline& & & & \vline& -\sqrt{{\frac{1}{3}}}\\
\hline
0 &\vline &0 &\dots &-\sqrt{{\frac{1}{3}}} & \vline& \mu\\
\end{array}\right|_{n-1.}$
 $\displaystyle{=\mu \Biggl\{\mu(-1)^{2n-2} \left| C_{n-2}\right|-(-1)^{2n-3} \sqrt{{\frac{1}{3}}} \left| \begin{array}{cccccc}
	& &  &  &\vline& 0\\ 
	& & C_{n-3} &  &\vline&0\\
	& &   & & \vline&\vdots\\
	& & & & \vline& -\sqrt{{\frac{1}{3}}}\\
	\hline
	0 &0  &\dots &-\sqrt{{\frac{1}{2}}} & \vline& \mu\\
\end{array} \right|_{\tiny{n-2 }}\Biggl\} }$\\
$ + \sqrt{{\frac{1}{3}}} \Biggl\{\mu(-1)^{2n-2}
\left| \begin{array}{cccc}
	-\sqrt{{\frac{1}{3}}} &\vline &-\sqrt{{\frac{1}{2}}} &0  \\ 
	\hline
	0&\vline & &   \\
	\vdots&\vline & & C_{n-3}  \\
	0 &\vline  &&  \\
\end{array} \right|_{\tiny{n-2}} (-1)^{2n-4}\sqrt{{\dfrac{1}{3}}} \left| \begin{array}{cccccc}
	-\sqrt{{\frac{1}{3}}} &\vline &-\sqrt{{\frac{1}{2}}} &0 &\vline &0\\ 
	\hline
	0& \vline& C_{n-4} &  &\vline&0\\
	\vdots& \vline&   & & \vline &\vdots\\
	\hline
	0 &\vline  &\dots &-\sqrt{{\frac{1}{2}}} & \vline& -\sqrt{{\frac{1}{3}}}\\
\end{array} \right|_{\tiny{n-2}}\Biggl\} $

$ =\mu \Biggl\{\mu(-1)^{2n-2} \left| C_{n-2}\right| - (-1)^{2n-3} \sqrt{{\frac{1}{3}}}
\biggr[ (-1)^{2n-4} \bigl(-\sqrt{{\frac{1}{3}}}\bigl) \left| C_{n-3}\right| \\ - (-1)^{2n-5} \sqrt{{\frac{1}{2}}}
 \left| \begin{array}{cccccc}
	\mu &-\sqrt{{\frac{1}{2}}} &0  &\dots  &0 & 0\\ 
	-\sqrt{{\frac{1}{2}}}&\mu & -\sqrt{{\frac{1}{2}}} & 0 &\dots &0\\
	0& 	-\sqrt{{\frac{1}{2}}}& \mu &-\sqrt{{\frac{1}{2}}}  &\vdots &\vdots\\
	\vdots& 0&\ddots   &\ddots &\ddots  &0\\
	0&\vdots &\dots &\ddots &\mu & 0\\
	0 &0  &\dots &\dots &-\sqrt{{\frac{1}{2}}} & 0\\
\end{array} \right|_{\tiny{n-3}}\biggr]\Biggl\} $

$ + \sqrt{{\frac{1}{3}}} \Biggl\{\mu(-1)^{2n-2} \biggr[\bigl(-\sqrt{{\frac{1}{3}}}\bigl) \left| C_{n-3}\right|-\bigl(-\sqrt{{\frac{1}{2}}}\bigl) \left| \begin{array}{cccccc}
	0 &-\sqrt{{\frac{1}{2}}} &0  &\dots  &0 & 0\\ 
	0&\mu & -\sqrt{{\frac{1}{2}}} & 0 &\dots &0\\
	\vdots& -\sqrt{{\frac{1}{2}}}& \mu &-\sqrt{{\frac{1}{2}}}  &\vdots &\vdots\\
	\vdots& 0&\ddots   &\ddots &\ddots  &0\\
	0&\vdots &\dots &\ddots &\mu & -\sqrt{{\frac{1}{2}}}\\
	0 &0  &\dots &\dots &-\sqrt{{\frac{1}{2}}} & \mu\\
\end{array} \right|_{\tiny{n-3}}\biggr]\\
+ (-1)^{2n-2} \bigl(-\sqrt{{\frac{1}{3}}}\bigl) \biggr[ \bigl(-\sqrt{{\frac{1}{3}}}\bigl)
 \left| \begin{array}{cccccc}
 	& &  &  &\vline~~~~& 0\\ 
 	& &  &  &\vline~~~~&0\\
 	& & C_{n-4} &  &\vline~~~~&0\\
 	& &   & & \vline~~~~ &\vdots\\
 	& & & & \vline~~~~& 0\\
 	\hline
 	0 &0  &\dots &-\sqrt{{\frac{1}{2}}} & \vline~~~~& -\sqrt{{\frac{1}{3}}}\\
 \end{array} \right|_{\tiny{n-3 }}\\ -\bigl(-\sqrt{{\frac{1}{2}}}\bigl)
 \left| \begin{array}{cccccc}
 	0 &-\sqrt{{\frac{1}{2}}} &0  &\dots  &0 & 0\\ 
 	0&\mu & -\sqrt{{\frac{1}{2}}} & 0 &\dots &0\\
 	\vdots& -\sqrt{{\frac{1}{2}}}& \mu &-\sqrt{{\frac{1}{2}}}  &\vdots &\vdots\\
 	\vdots& 0&\ddots   &\ddots &\ddots  &0\\
 	0&\vdots &\dots &\ddots &\mu & 0\\
 	0 &0  &\dots &\dots &-\sqrt{{\frac{1}{2}}} & -\sqrt{{\frac{1}{3}}}\\
 \end{array} \right|_{{n-3 }}
 \biggr]\Biggl\}$\\
 After simplification,\\
$ =\mu \Biggl\{\mu(-1)^{2n-2} \left| C_{n-2}\right| + \bigl(\frac{1}{3}\bigl) (-1)^{2n-3} (-1)^{2n-4} \left| C_{n-3}\right|\Biggl\} $\\ $+ \sqrt{{\frac{1}{3}}} \Biggl\{\mu(-1)^{2n-2} \biggr[ \bigl(-\sqrt{{\frac{1}{3}}}\bigl) \left| C_{n-3}\right|\biggr] + (-1)^{2n-3} \bigl(-\sqrt{{\frac{1}{3}}}\bigl) 
\biggr[\bigl(-\sqrt{{\frac{1}{3}}}\bigl) 
 \left| \begin{array}{cccccc}
	& &  &  &\vline& 0\\ 
	& &  &  &\vline&0\\
	& & C_{n-4} &  &\vline&0\\
	& &   & & \vline&\vdots\\
	& & & & \vline& 0\\
	\hline
	0 &0  &\dots &-\sqrt{{\frac{1}{2}}} & \vline& -\sqrt{{\frac{1}{3}}}\\
\end{array} \right|_{{n-3}}\biggr]\Biggl\} $\\
$ =\mu \Biggl\{\mu(-1)^{2n-2} \left| C_{n-2}\right| + \bigl(\frac{1}{3}\bigl) (-1)^{2n-3} (-1)^{2n-4} \left| C_{n-3}\right|\Biggl\} $\\
$+ \sqrt{{\frac{1}{3}}} \Biggl\{\mu(-1)^{2n-2} \bigl(-\sqrt{{\frac{1}{3}}}\bigl) \left| C_{n-3}\right| + (-1)^{2n-3} \bigl(\frac{1}{3}\bigl)
\biggr[\bigl (-1)^{2n-6} \bigl(-\sqrt{{\frac{1}{3}}}\bigl) \left| C_{n-4}\right| \\
+ (-1)^{2n-7} \bigl(-\sqrt{{\frac{1}{2}}}\bigl) 
\left| \begin{array}{cccccc}
	\mu &-\sqrt{{\frac{1}{2}}} &0  &\dots  &0 & 0\\ 
	-\sqrt{{\frac{1}{2}}}&\mu & -\sqrt{{\frac{1}{2}}} & 0 &\dots &0\\
	0& 	-\sqrt{{\frac{1}{2}}}& \mu &-\sqrt{{\frac{1}{2}}}  &\vdots &\vdots\\
	\vdots& 0&\ddots   &\ddots &\ddots  &0\\
	0&\vdots &\dots &\ddots &\mu & 0\\
	0 &0  &\dots &\dots &-\sqrt{{\frac{1}{2}}} & 0\\
\end{array} \right|_{\small{n-4}}
\biggr]\Biggl\} $\\
$ =\mu \Biggl\{\mu(-1)^{2n-2} \left| C_{n-2}\right| + \bigl(\frac{1}{3}\bigl) (-1)^{2n-3} (-1)^{2n-4} \left| C_{n-3}\right|\Biggl\} $\\
$+ \sqrt{{\frac{1}{3}}} \Biggl\{\mu(-1)^{2n-2} \bigl(-\sqrt{{\frac{1}{3}}}\bigl) \left| C_{n-3}\right| + (-1)^{2n-3} \bigl(\frac{1}{3}\bigl)
\bigl (-1)^{2n-6} \bigl(-\sqrt{{\frac{1}{3}}}\bigl) \left| C_{n-4}\right|
\Biggl\} $\\
$ =\mu \Biggl\{\mu(-1)^{2n-2} \varOmega_{n-2} + \bigl(\frac{1}{3}\bigl) (-1)^{2n-3} (-1)^{2n-4} \varOmega_{n-3}\Biggl\} $\\
$+ \sqrt{{\frac{1}{3}}} \Biggl\{\mu(-1)^{2n-2} \bigl(-\sqrt{{\frac{1}{3}}}\bigl) \varOmega_{n-3} - (-1)^{2n-3} \bigl(\frac{1}{3}\bigl)
\bigl (-1)^{2n-6} \bigl(\sqrt{{\frac{1}{3}}}\bigl) \varOmega_{n-4}
\Biggl\} $\\

$ =(-1)^{2n-2}\mu^{2} \varOmega_{n-2} + (-1)^{4n-7} \bigl(\frac{1}{3}\bigl) \mu \varOmega_{n-3} $
$- (-1)^{2n-2} \bigl(\frac{1}{3}\bigl) \mu \varOmega_{n-3} - (-1)^{4n-9} \bigl(\frac{1}{9}\bigl)  \varOmega_{n-4}
 $\\
 Therefore,
 $\phi\big(P_n(G) : \mu \big ) = \mu^{2} \varOmega_{n-2} -\frac{2}{3} \mu \varOmega_{n-3} + \frac{1}{9} \varOmega_{n-4}.$
\end{proof}

\section {ABS-Eigenvalues of Some Classes of Graphs} 
In this section, we obtain $ABS$-eigenvalues of some classes of graphs like $r$-regular, $K_{m,n}$, and $S_n$ graphs. We also obtain some properties of the $ABS$ matrix of a graph.  The following Lemmas  gives the eigenvalues of $A(G)$ related with eigenvalues of adjacency matrix of graphs $K_{n}$, $K_{m,n}$, $C_{n}$, and $S_n$.
It is known that, if all entries of $A$ are strictly positive, then we say $A$ is positive matrix we write $(A)_{ij} > 0$.
If $A$ is a real and symmetric matrix, then all eigenvalues of $A$ are real.
We say that as eigenvalue is simple if its algebraic multiplicity is $1$. As the $ABS$ matrix is an irreducible non-negative symmetric real matrix, its all eigenvalues are real. Also, the trace of the $ABS$ matrix is $0$. 
\par Recall the following results in sequel.
\begin{lem}[Bapat \cite{Bap}]\label{1}
For any positive integer $n$, the eigenvalues of complete graph $K_{n}$ are $\lambda_{1} = n-1$ and $\lambda_{2} = \dots = \lambda_{n} = -1$.
\end{lem}
\begin{lem}[Bapat \cite{Bap}]\label{2}
For any positive integer $m,n$, the eigenvalues of complete 
bipartite graph $K_{m,n}$ are $\lambda_{1} = \sqrt{mn}$, $\lambda_{2} =  \lambda_{3} = \dots = \lambda_{m+n-1} = 0$ and $\lambda_{m+n} = -\sqrt{mn}$.
\end{lem}
\begin{lem}[Bapat \cite{Bap}]\label{3}
For $n \geq 2$, the eigenvalues of $G = C_{n}$ are $\lambda_{i} = 2 cos\frac{2\pi i}{n}$, where $i = 1, \dots, n$.
\end{lem}
\begin{lem}[Brouwer and Haemers \cite{Brouwer AE}]\label{4}
For any positive integer $n$, the eigenvalues of star graph $S_{n}$ are $\lambda_{1} = \sqrt{n-1}$, $\lambda_{2} = \dots = \lambda_{n-1}=0$ and $\lambda_{n} = -\sqrt{n-1}$.
\end{lem}
In the following result, we give the $ABS$ eigenvalues of a $r$-regular graph.
\begin{thm}
	Let $G$ be a connected graph of order $n\geq 3$ with $\lambda_{i} , 1 \leq i \leq n$ an eigenvalues of $A(G)$ and let $\mu_{i},  1 \leq i \leq n$ be its $ABS$ eigenvalues.
	If $G$ is a  $r$-regular graph, then the $ABS$ eigenvalues of $G$ are $\mu_{i} = \dfrac{\sqrt{r^{2}-r}}{r}$ $\lambda_{i}$ for $i = 1, 2, \dots, n$.
\end{thm}
\begin{proof}
	Let $G$ be a connected $r$-regular graph of order $n \geq 3$ with an eigenvalues $\lambda_{1},\lambda_{2},\dots, \lambda_{n}.$ Observe that, $\tilde{A}(G)$ $=$ $\dfrac{\sqrt{r^{2}-r}}{r}$ $A(G)$. This implies that, the $ABS$ eigenvalues of $G$ are  $\mu_{i} = \dfrac{\sqrt{r^{2}-r}}{r}$ $\lambda_{i}$, for $i = 1, 2, \dots, n$.
\end{proof}
\begin{corollary}
If $G = K_n$, then the $ABS$ eigenvalues of are $\mu_{1} = (n-1)\sqrt{\frac{n-2}{n-1}}$ and $\mu_{2} = \dots = \mu_{n} = -\sqrt{\frac{n-2}{n-1}}$. 
\end{corollary}
\begin{proof} 
By the Lemma \ref{1}, the eigenvalues of $K_n$ are $\lambda_{1} = n-1$ and $\lambda_{2} = \dots = \lambda_{n} = -1$. Observe that, $\tilde{A}(K_n)$ $=$ $\sqrt{\dfrac{n-2}{n-1}}A(K_n)$. Therefore, the $ABS$ eigenvalues of $K_n$ are  $\mu_{1} = (n-1)\sqrt{\frac{n-2}{n-1}}$ and $\mu_{2} = \dots = \mu_{n} = -\sqrt{\frac{n-2}{n-1}}$.
\end{proof} 
\begin{corollary}
The $ABS$ eigenvalues $C_n$ are $\mu_{i}$ $=$ $\sqrt{2}\left(\cos\dfrac{2\pi i}{n}\right)$, where $i = 0, 
\dots, n-1$. 
\end{corollary}
\begin{proof} 
By the Lemma \ref{3}, the eigenvalues of  $C_n$ are $\lambda_{i} = 2 cos\frac{2\pi i}{n}$, for $i = 1, \dots, n$. Observe that, $\tilde{A}(C_n)$ $=$ $\dfrac{1}{\sqrt 2}A(C_n)$. Therefore, the $ABS$ eigenvalues of $C_n$ are $\mu_{i}$ $=$ $\sqrt{2}\left(\cos\dfrac{2\pi i}{n}\right)$, where $i = 0, 
\dots, n-1$. 
\end{proof}
The following result gives the $ABS$ eigenvalues of complete bipartite graph.
\begin{thm}\label{15}
Let $m,n$ be positive integers.  The $ABS$ eigenvalues of $K_{m,n}$ are $\sqrt{\frac{mn(m+n-2)}{m+n}},$ $-\sqrt{\frac{mn(m+n-2)}{m+n}}$ and $\underbrace{0,\ldots, 0}_{(m+n-2)\textrm{-times}}.$
\end{thm}
\begin{proof}
Follows by the Theorem \ref{charofabsKmn}.
\end{proof}
The following result gives the $ABS$ eigenvalues of star graph.
\begin{thm}
Let $n$ be a positive integer. The $ABS$ eigenvalues of star graph $S_{n}$ are $ \sqrt{\dfrac{(n-1)(n-2)}{n}},-\sqrt{\dfrac{(n-1)(n-2)}{n}}$ and $\underbrace{0,\ldots, 0}_{(n-2)\textrm{-times}}.$
\end{thm}
\begin{proof}
It is easy to observe that the $\tilde{A}(S_{n})$ $=$ $\sqrt{\frac{n-2}{n}}$ $A(G(S_{n}))$. From Lemma \ref{4}, the eigenvalues of $A(S_{n})$  are
$ \sqrt{n-1}, -\sqrt{n-1}$ and $\underbrace{0,\ldots, 0}_{(n-2)\textrm{-times}}.$ Therefore, the $ABS$ eigenvalues of star graph $S_{n}$ are $ \sqrt{\dfrac{(n-1)(n-2)}{n}},-\sqrt{\dfrac{(n-1)(n-2)}{n}}$ and $\underbrace{0,\ldots, 0}_{(n-2)\textrm{-times}}.$
\end{proof}
The following result gives the relation between sum of squares of all eigenvalues of $ABS$ matrix of $r$-regular graph and the modified second Zagreb index of that graph.
\begin{thm} \label{11}
Let $G$ be a connected graph of order $n\geq 4$ with $\mu_{1},\mu_{2},\ldots, \mu_{n}$  the eigenvalues of $\tilde{A}(G)$. Then 
$\displaystyle{\sum^{n}_{i=1}\big(\mu_{i}\big)^{2} \leq \big(n-1\big) \biggl(n-2R_{-1}(G)\biggl),}$ the equality hold if and only if $G$ is a $r$-regular graph.
\end{thm}
\begin{proof}
Let $G$ be a connected graph of order $n\geq 4$ and let $\mu_{1},\mu_{2},\ldots, \mu_{n}$ be the eigenvalues of $\tilde{A}(G).$ As the trace of $\tilde{A}(G)$ equal to $0,$ we have  $\displaystyle{\sum^{n}_{i=1}\mu_{i}=0}$. Moreover, 
\begin{align*}
\sum^{n}_{i=1}\mu_{i}^{2} = trace \big(\tilde{A}(G)\big)^{2} \leq \big(n-1\big) \sum_{v_{i}v_{j}\in E(G)}\biggl(\frac{1}{d_{i}}+\frac{1}{d_{j}}-\frac{2}{d_{i}d_{j}}\biggl) \leq \big(n-1\big) \biggl(n-2R_{-1}(G)\biggl).
\end{align*}
Let $G$ be a $r$-regular graph. Then degree of every vertex is $n-1$, which gives  
\begin{align*}
\sum^{n}_{i=1}\mu_{i}^{2} = trace \big(\tilde{A}(G)\big)^{2} \leq \big(n-1\big) \sum_{v_{i}v_{j}\in E(G)}\biggl(\frac{1}{d_{i}}+\frac{1}{d_{j}}-\frac{2}{d_{i}d_{j}}\biggl) = r \biggl(n-2R_{-1}(G)\biggl).
\end{align*}
\end{proof}
The following result gives the relation between sum of the squares of all eigenvalues of $ABS$ matrix of a graph and the harmonic index of that graph.
\begin{thm}
Let $G$ be a connected graph of order $n\geq 3$ with $m$ edges and let $\mu_{1},\mu_{2},\dots,\mu_{n}$ be the eigenvalues of $\tilde{A}(G)$. Then
\begin{center}
$\displaystyle{\sum^{n}_{i=1}\big(\mu_{i}\big)^{2} = 2\biggl(m-H(G)\biggl)}$ and 
$\displaystyle{\sum^{n}_{1\leq i < j \leq n}\mu_{i} \mu_{j} = \biggl(H(G)-m\biggl)}
,$
\end{center} where $H(G) = \displaystyle{\sum_{v_{i}v_{j}\in E(G)}\frac{2}{d_{i}+d_{j}}}
$ is the harmonic index of a graph $G.$
\end{thm}
\begin{proof}
Let $G$ be a connected graph of order $n\geq 3$ and let $\mu_{1},\mu_{2},\ldots, \mu_{n}$ be the eigenvalues of $\tilde{A}(G).$ As the trace of $\tilde{A}(G)$ equal to $0,$ we have  $\displaystyle{\sum^{n}_{i=1}\mu_{i}=0}$.
Therefore,
\begin{align*}
\sum^{n}_{i=1}\big(\mu_{i}\big)^{2} = trace \big(\tilde{A}(G)\big)^{2} = 2 \sum_{v_{i}v_{j}\in E(G)}\biggl(\frac{d_{i}}{d_{i}+d_{j}}+\frac{d_{j}}{d_{i}+d_{j}}-\frac{2}{d_{i}+d_{j}}\biggl) = 2\biggl(m -\sum_{v_{i}v_{j}\in E(G)}\frac{2}{d_{i}+d_{j}}\biggl).
\end{align*}
Hence, 
\begin{align*}
\sum^{n}_{i=1}\big(\mu_{i}\big)^{2} = 2\biggl(m-H(G)\biggl).
\end{align*}
Moreover,
\begin{align*}
\sum^{n}_{1\leq i < j \leq n}\mu_{i} \mu_{j} = \frac{1}{2} \biggl(\biggl(\sum^{n}_{i=1}\mu_{i}\biggl)^{2} - \sum^{n}_{i=1}\mu_{i}^{2}\biggl) = \biggl(H(G)-m\biggl).
\end{align*}
\end{proof}

\section {ABS-Energy of k-Splitting and k-Shadow Graphs}
In this section, we obtain the $ABS$ energy of $k$-splitting graph and $k$-shadow graph of a $r$- regular graphs in term of energy of graph. These graphs defined by Vaidya and Popat\cite{Vaidya} in 2017. 
Recently, the $ABC$, Sombor, $ISI$, $SDD$ and Randic energies of $k$-splitting graph and $k$-shadow graph  related results obtained and studied by various authors.   
 Recall the following definitions and Lemma in sequel.
\begin{defn} [Vaidya and Popat\cite{Vaidya}] 
 Let $G$ be a graph with $n$ vertices and $m$ edges. The $k$-splitting graph of a graph $G$ is denoted by $Spl_{k}(G)$. It is obtained by adding to each vertex $x$ of $G$ new $k$ vertices, $x_1, x_2, \dots, x_k$ such that $x_i$, $1 \leq i \leq k$ is adjacent to each vertex of that is adjacent to $x$ in $G$.
\end{defn}

\begin{defn} [Vaidya and Popat\cite{Vaidya}]
 Let $G$ be a connected graph with $n$ vertices and $m$ edges. The $k$-shadow graph of a graph $G$ is denoted by $D_{k}(G)$. It is obtained by taking $k$ copies of $G$, say $G_1, G_2, \dots, G_k$ then join each vertex $x$ in $G_i$ to the neighbors of the corresponding vertex $y$ in $G_j$, such that $1 \leq i, j \leq k$.
\end{defn}
\begin{defn} [Bapat \cite{Bap}] \label{16}
Let $A$ and $B$ be two real matrices of order $m \times n$ and $p \times q$ respectively. Then the Kronecker product of is denoted by $A \otimes B$ and defined as\\
$A \otimes B$ $= \left( \begin{array}{cccc}
a_{11}B & a_{12}B & \dots & a_{1n}B  \\ 
a_{21}B & a_{22}B & \dots & a_{2n}B \\
 \vdots& \vdots & \ddots & \vdots  \\
 a_{m1}B & a_{m1}B & \dots & a_{mn}B  \\
\end{array} \right)_{mp \times nq.}$ 
\end{defn}
\begin{lem} [Bapat \cite{Bap}] \label{14}
Let $A$ and $B$ be symmetric matrices of order $m \times m$ and $n \times n$,  respectively. If $\xi_1, \dots, \xi_m$ and  $\eta_1, \dots, \eta_n$ are the eigenvalues of $A$ and $B$. Then the eigenvalues of $A \otimes B$ are $\xi_{i}\eta_{j}$, where $i = 1, \dots, m$ and $j = 1, \dots, n$.
\end{lem}

\begin{thm}
Let $G$ be a $r$-regular graph on $n$ vertices and $m$ edges. Then $E_{ABS}(Spl_{k}(G))$ $= \sqrt{\dfrac{5rk^{2} + 15rk - 9k + 10r - 10}{r(k+1)(k+2)}} E_{A}(Spl_{k}(G))$.
\end{thm}
\begin{proof}
Let $G$ be a $r$-regular graph with $n$ vertices say $v_1, \dots, v_n$ and $m$ edges. Let $Spl_{k}(G)$ be $k$-splitting graph of $G$. Now, by constructing the $k$-splitting graph of $G,$  $Spl_{k}(G)$ has the $(k+1)n$ vertices. Let $\mu_1, \dots, \mu_{n+nk}$ be eigenvalues of $Spl_{k}(G)$.  Then the $ABS$ matrix of $Spl_{k}(G)$ is
$\tilde{A}(Spl_{k}(G))$ $= \left( \begin{array}{cccc}
C_1 & C_2 & \dots & C_2  \\ 
C_2 & 0 & \dots & 0 \\
 \vdots& \vdots & \ddots & \vdots  \\
 C_2 & 0 & \dots & 0  \\
\end{array} \right)_{n+kn.}$ \\

Where the $C_1=[a_{ij}]_{n \times n}$ and $C_2=[b_{ij}]_{n \times n}$ are  symmetric matrices such that
\begin{equation*}
 C_1=[a_{ij}] = \begin{cases}\sqrt{\frac{2r(k+1)-2}{2r(k+1)}}\hspace{0.5cm}   \text{if}~~v_{i} v_{j}\in E(G),\\
 0\hspace{1.9cm} \text{otherwise}.
 \end{cases}
 \end{equation*}
 and 
\begin{equation*}
 C_2=[b_{ij}] = \begin{cases}\sqrt{\frac{r(k+2)-2}{r(k+2)}}\hspace{0.5cm}   \text{if}~~v_{i} v_{j}\in E(G),\\
 0\hspace{1.9cm} \text{otherwise}.
 \end{cases}
 \end{equation*}\\
Therefore,
\begin{eqnarray*} 
\tilde{A}(Spl_{k}(G))&=& \left( \begin{array}{cccc}
\sqrt{\frac{2r(k+1)-2}{2r(k+1)}} A(G) & \sqrt{\frac{r(k+2)-2}{r(k+2)}} A(G) & \dots & \sqrt{\frac{r(k+2)-2}{r(k+2)}} A(G)  \\ 
\sqrt{\frac{r(k+2)-2}{r(k+2)}} A(G) & 0 & \dots & 0 \\
 \vdots& \vdots & \ddots & \vdots  \\
 \sqrt{\frac{r(k+2)-2}{r(k+2)}} A(G) & 0 & \dots & 0  \\
\end{array} \right)_{n+kn.}\\
&=& \left( \begin{array}{cccc}
\sqrt{\frac{2r(k+1)-2}{2r(k+1)}}  & \sqrt{\frac{r(k+2)-2}{r(k+2)}}  & \dots & \sqrt{\frac{r(k+2)-2}{r(k+2)}}   \\ 
\sqrt{\frac{r(k+2)-2}{r(k+2)}}  & 0 & \dots & 0 \\
 \vdots& \vdots & \ddots & \vdots  \\
 \sqrt{\frac{r(k+2)-2}{r(k+2)}}  & 0 & \dots & 0  \\
\end{array} \right)_{k+1} \otimes~~\big[A(G)\big]_{n}\\
 &=& D \otimes~~A(G),
\end{eqnarray*} where 
\begin{eqnarray*}
	D&=& \left( \begin{array}{cccc}
 \sqrt{\frac{2r(k+1)-2}{2r(k+1)}}  & \sqrt{\frac{r(k+2)-2}{r(k+2)}}  & \dots & \sqrt{\frac{r(k+2)-2}{r(k+2)}}   \\ 
 \sqrt{\frac{r(k+2)-2}{r(k+2)}}  & 0 & \dots & 0 \\
  \vdots& \vdots & \ddots & \vdots  \\
  \sqrt{\frac{r(k+2)-2}{r(k+2)}}  & 0 & \dots & 0  \\
 \end{array} \right)_{k+1}\\
 &=&\left( \begin{array}{cccc}
 \sqrt{\frac{2r(k+1)-2}{2r(k+1)}} I_{1 \times 1}  & \sqrt{\frac{r(k+2)-2}{r(k+2)}}J_{1 \times k} \\
 \sqrt{\frac{r(k+2)-2}{r(k+2)}} J_{k \times 1}  & O_{k \times k} \\
 \end{array} \right)_{k+1.}
\end{eqnarray*}
 Where, the $J$ and $O$ be the ones and zeros matrices. Observe that the rank of matrix $D$ is two.  Therefore, it has exactly two non-zero eigenvalues and remaining eigenvalues are $0$ with multiplicity $k-1$.
 The characteristic polynomial of the matrix $D$ is
 $det(\xi I - D)=\left| \begin{array}{cccc}
(\xi -\sqrt{\frac{2r(k+1)-2}{2r(k+1)}}) I_{1 \times 1}  & -\sqrt{\frac{r(k+2)-2}{r(k+2)}}J_{1 \times k} \\
-\sqrt{\frac{r(k+2)-2}{r(k+2)}} J_{k \times 1}  & \xi I_{k \times k} \\
\end{array} \right|.$
Now, by Lemma {\ref{6}} we have,
\begin{eqnarray*}
det\bigl(\xi I - D\bigl)&=& \xi^{k} \Bigg|\biggl(\xi -\sqrt{\frac{2r(k+1)-2}{2r(k+1)}}\biggl)I_{1 \times 1}-\biggl({\sqrt{\frac{r(k+2)-2}{r(k+2)}}}\biggl)^{2}J_{1 \times k}~~ \frac{I_{k \times k}}{\xi} J_{k \times 1}\Bigg|\\
&=& \xi^{k-1} \Bigg|\xi^{2} -\sqrt{\frac{2r(k+1)-2}{2r(k+1)}}\xi-\biggl({\sqrt{\frac{r(k+2)-2}{r(k+2)}}}\biggl)^{2}\Bigg|.
\end{eqnarray*}
After simplification, we have
\begin{center}
$\det\bigl(\xi I - D\bigl)$ $= \xi^{k-1} \biggl(\xi^{2}-\dfrac{1}{2}\left(\sqrt{\frac{2r(k+1)-2}{2r(k+1)}}\pm \sqrt{\frac{5rk^{2}+15rk-9k+10r-10}{r(k+1)(k+2)}}\right)\biggl).$
\end{center}
Therefore, the two non-zero eigenvalues of matrix $D$ are\\ $\xi_1 =\dfrac{1}{2}\biggl(\sqrt{\frac{2r(k+1)-2}{2r(k+1)}}+\sqrt{\frac{5rk^{2}+15rk-9k+10r-10}{r(k+1)(k+2)}}\biggl)$ and $\xi_2 = \dfrac{1}{2}\biggl(\sqrt{\frac{2r(k+1)-2}{2r(k+1)}}-\sqrt{\frac{5rk^{2}+15rk-9k+10r-10}{r(k+1)(k+2)}}\biggl)$.
Let $\lambda_1, \dots, \lambda_n$ be the eigenvalues of $A(G)$. Then by Lemma {\ref{14}}, the eigenvalues of $\tilde{A}(Spl_{k}(G))$ are $\xi_i \lambda_j$ where $i = 1, \dots, k+1$ and $j = 1, \dots, n$.
For $k \geq 0$, observe that
\begin{center}
 $\sqrt{\dfrac{5rk^{2}+15rk-9k+10r-10}{r(k+1)(k+2)}} \geq \sqrt{\dfrac{2r(k+1)-2}{2r(k+1)}}$.
 \end{center} Then $\left|\dfrac{\sqrt{\frac{2r(k+1)-2}{2r(k+1)}}-\sqrt{\frac{5rk^{2}+15rk-9k+10r-10}{r(k+1)(k+2)}}}{2}\right| = \dfrac{\sqrt{\frac{5rk^{2}+15rk-9k+10r-10}{r(k+1)(k+2)}}-\sqrt{\frac{2r(k+1)-2}{2r(k+1)}}}{2}$. Thus,\\
$E_{ABS}(Spl_{k}(G))=\sum^{k+1}_{i=1} |\xi_i \lambda_j|=\sum^{n+kn}_{j=1} \left|\frac{(\sqrt{\frac{2r(k+1)-2}{2r(k+1)}}\pm \sqrt{\frac{5rk^{2}+15rk-9k+10r-10}{r(k+1)(k+2)}})}{2} \lambda_{j}\right|$\\$
= \sum^{n}_{j=1} |\lambda_{j}| \left[\frac{\sqrt{\frac{2r(k+1)-2}{2r(k+1)}}+ \sqrt{\frac{5rk^{2}+15rk-9k+10r-10}{r(k+1)(k+2)}}}{2}  \frac{-\sqrt{\frac{2r(k+1)-2}{2r(k+1)}}+ \sqrt{\frac{5rk^{2}+15rk-9k+10r-10}{r(k+1)(k+2)}}}{2}\right]$\\$
= \sqrt{\frac{5rk^{2} + 15rk - 9k + 10r - 10}{r(k+1)(k+2)}} E_{A}(Spl_{k}(G)).$  
\end{proof}
\begin{thm}
	Let $G$ be a $r$-regular graph on $n$ vertices and $m$ edges. Then $E_{ABS}(D_{k}(G)) = k \sqrt{1-\frac{1}{kr}} E_{A}(D_{k}(G))$.
\end{thm}
\begin{proof}
	Let $G$ be a $r$-regular graph with $n$ vertices say $v_1, \dots, v_n$ and $m$ edges. Let $D_{k}(G)$ be $k$-shadow graph of $G$. Now, by constructing the $k$-shadow graph of $G, D_{k}(G)$ has the $kn$ vertices. Then the $ABS$ matrix of $D_{k}(G)$ is
	$\tilde{A}(D_{k}(G))$ $= \left( \begin{array}{cccc}
		\tilde{A}(G)& \tilde{A}(G) & \dots & \tilde{A}(G)  \\ 
		\tilde{A}(G) & \tilde{A}(G) & \dots & \tilde{A}(G) \\
		\vdots& \vdots & \ddots & \vdots  \\
		\tilde{A}(G) & \tilde{A}(G) & \dots & \tilde{A}(G)  \\
	\end{array}\right)_{kn}$
	By definition \ref{16}, we have
\begin{eqnarray*}
\tilde{A}(D_{k}(G))&=&\left( \begin{array}{cccc}
		\sqrt{\frac{2kr-2}{2kr}}& \sqrt{\frac{2kr-2}{2kr}} & \dots & \sqrt{\frac{2kr-2}{2kr}}  \\ 
		\sqrt{\frac{2kr-2}{2kr}} & \sqrt{\frac{2kr-2}{2kr}} & \dots & \sqrt{\frac{2kr-2}{2kr}} \\
		\vdots& \vdots & \ddots & \vdots  \\
		\sqrt{\frac{2kr-2}{2kr}} & \sqrt{\frac{2kr-2}{2kr}} & \dots & \sqrt{\frac{2kr-2}{2kr}}  \\
	\end{array}\right)_{k} 
	\otimes  \big[A(G)\big]_{n}\\
	&=& \sqrt{\frac{2kr-2}{2kr}} \left( \begin{array}{cccc}
		1& 1 & \dots & 1  \\ 
		1& 1 & \dots & 1 \\
		\vdots& \vdots & \ddots & \vdots  \\
		1 & 1 & \dots & 1  \\
	\end{array}\right)_{k} 
	\otimes  \big[A(G)\big]_{n}\\
&=& \sqrt{\frac{2kr-2}{2kr}}\big[J\big]_{k} 
	\otimes  \big[A(G)\big]_{n,}
\end{eqnarray*}
	where $J$ is the ones matrix of order $k \times k$. Let $\lambda_1, \lambda_2, \dots, \lambda_n$ be the eigenvalues of $A(G)$ and $\gamma_1, \gamma_2 \dots, \gamma_k$ be the eigenvalues of the matrix $J$. The eigenvalues of $J_k$ are $k,\underbrace{0,\ldots, 0}_{(k-1)\textrm{-times}}.$ Therefore, by Lemma \ref{14}, eigenvalues of $\tilde{A}(D_{k}(G))$ are $\sqrt{\frac{2kr-2}{2kr}}$ $\gamma_i\lambda_j $, where $i = 1, 2, \dots, k$ and $j = 1, 2, \dots, n$. Then we have the $E_{ABS} (D_{k}(G)) = \sum^{n}_{i=1}|\mu_{i}|$ $=$  $\sum^{k}_{i=1}|\sqrt{\frac{2kr-2}{2kr}}\gamma_i\lambda_j|$ and $j = 1, 2, \dots, n$. Thus  
	$E_{ABS} (D_{k}(G)) = \sqrt{\frac{2kr-2}{2kr}} |\sum^{k}_{i=1}\gamma_i||\sum^{n}_{j=1}\lambda_j| = k \sqrt{\frac{2kr-2}{2kr}} |\sum^{n}_{j=1}\lambda_j| = k \sqrt{1-\frac{1}{kr}}  E_{A}(D_{k}(G))$.
\end{proof}


\begin{thebibliography}{10}

\bibitem{Aarthi K} Aarthi K, Elumalai S, Balachandran S, Mondal S., Extremal values of the atom-bond sum-connectivity index in bicyclic graphs,~\textit{Journal of Applied Mathematics and Computing.}~\textbf{69}(6)~2023,~4269-85.
\bibitem{Ali A2} Ali A, Gutman I, Furtula B, Red$\check{z}$epovi$\acute{c}$ I, Do$\check{s}$li$\acute{c}$ T, Raza Z., Extremal results and bounds for atom-bond sum-connectivity index,~\textit{MATCH Commun. Math. Comput. Chem.}~\textbf{92}~2024,~271-314.
\bibitem{Ali} Ali A, Furtula B, Red$\check{z}$epovi$\acute{c}$ I, Gutman I., Atom-bond sum-connectivity index,~\textit{Journal of Mathematical Chemistry.}~\textbf{60}~(10)~2022,~2081-2093.
\bibitem{Bap} Bapat R B, \textit{Graphs and matrices}, {New York: Springer}, (Vol. 27)~2010.
\bibitem{Brouwer AE} Brouwer AE, Haemers WH, \textit{Spectra of graphs}, {Springer Science and Business Media}~2011.
\bibitem{Doob} Cvetkovi$\acute{c}$ D. M., Doob M., Sach H., Spectra of Graphs, ~\textit{Academic Press, New York}, ~1980.	
\bibitem{E. Estrada3} Estrada E., Atom-bond connectivity and the energetic of branched alkanes,~\textit{Chem. Phys. Lett.}~\textbf {463}~2008,~422-425.
\bibitem{Est1} Estrada E, Torres L, Rodriguez L, Gutman I., An atom-bond connectivity index: modelling the enthalpy of formation of alkanes, ~\textit{NISCAIR-CSIR, India},~1998.
\bibitem{Samp} Sampathkumar E.  and Chikkodimath S, Semitotal graphs of a graph-I,~\textit{J. Karnatak Univ. Sci.},~\textbf{18}~1973,~274-280.
\bibitem{Harry} Harary F, \textit{Graph Theory}, Addison-Wesley, Reading, ~1969.
\bibitem{Fajtl} Fajtlowicz S, On conjectures of Graffti-II, ~\textit {Congr. Num.}~\textbf{60}~1987,~187-197.
\bibitem{Gut} Gutman I.,  Trinajsti$\acute{c}$ N., Graph theory and molecular orbital. Total -electron energy of alternant hydrocarbons,~\textit{Chemical Physics Letters}~\textbf{17}~(4)~1972,~535-538.
\bibitem{Hu Y} Hu Y, Wang F., On the maximum atom-bond sum-connectivity index of trees,~\textit{MATCH Commun. Math. Comput. Chem.}~\textbf{91}~2024,~709-23.
\bibitem{Jahanbani A} Jahanbani A, Red$\check{z}$epovi$\acute{c}$ I. On the generalized ABS index of graphs.~\textit{Filomat.}~\textbf{37}(30)~2023,~10161-9.	
\bibitem{Li F} Li F, Ye Q, Lu H., The greatest values for atom-bond sum-connectivity index of graphs with given parameters,~\textit{Discrete Applied Mathematics}~\textbf{344}~2024,~188-96.
\bibitem{Li F1} Li F, Ye Q., Extremal graphs with given parameters in respect of general ABS index,~\textit{Applied Mathematics and Computation}~\textbf{482}~2024,~128974.	
\bibitem{Lin} Lin Z, Zhou T, Liu Y., On ABS Estrada index of trees, ~\textit{Journal of Applied Mathematics and Computing}~2024,~1-13 .
\bibitem{Nik} Nikoli$\acute{c}$ S, Kova$\check{c}$evi$\acute{c}$ G, Mili$\check{}$evi$\acute{c}$ A, Trinajsti$\acute{c}$ N., The Zagreb indices 30 years after,~\textit {Croatica chemica acta}~\textbf{76}(2)~2003,~113-124.
\bibitem{Nithya P} Nithya P, Elumalai S, Balachandran S, Mondal S., Smallest ABS index of unicyclic graphs with given girth,~\textit{Journal of Applied Mathematics and Computing}~\textbf{69}(5)~2023,~3675-92.
\bibitem{Randi} Randi$\acute{c}$ M., On characterization of molecular branching, ~\textit{J. Am. Chem. Soc.}~\textbf {97}~1975,~6609–6615.
\bibitem{Vaidya} Vaidya Samir K., and Kalpesh M. Popat. Energy m-splitting and m-shadow graphs, ~\textit {Far East Journal of Mathematical Sciences}~\textbf {102}~(8),~2017,~1571-1578.
\bibitem{Zhang Y} Zhang Y, Wang H, Su G, Das KC., Extremal problems on the Atom-bond sum-connectivity indices of trees with given matching number or domination number,~\textit{Discrete Applied Mathematics}~\textbf{345}~2024,~190-206.
\bibitem{Zuo X} Zuo X, Jahanbani A, Shooshtari H., On the atom-bond sum-connectivity index of chemical graphs,~\textit{Journal of Molecular Structure}~\textbf{1296}~2024,~136849.

	
	
	
	
	
\end{thebibliography}

Acknowledgment- The second author is supported by SPPU Pune under the ASPIRE project- 20TEC001284
\end{document}